\documentclass[12pt,a4paper]{article}

\usepackage{amssymb}

\usepackage{graphicx,amssymb,amsfonts,epsfig,amsthm,a4,amsmath,url}

\usepackage[latin1]{inputenc}

\newtheorem{thm}{Theorem}[section]

\newtheorem{lem}[thm]{Lemma}

\newtheorem{prop}[thm]{Proposition}

\theoremstyle{definition}

\newtheorem{defn}[thm]{Definition}

\theoremstyle{remark}

\newtheorem{rem}[thm]{Remark}

\numberwithin{equation}{section}

\newcommand{\Z}{\mathbf{Z}}

\newcommand{\R}{\mathbf{R}}

\newcommand{\C}{\mathbf{C}}

\newcommand{\Q}{\mathbf{Q}}

\newcommand{\K}{\mathbf{K}}

\begin{document}



\title{On lengths on semisimple groups}
\author{Yves de Cornulier}%
\date{May 21, 2009}
\maketitle

\begin{abstract}We prove that every length on a simple group over a locally compact field, is either bounded or proper.\end{abstract}

\section{Introduction}

Let $G$ be a locally compact group.
We call here a \textit{semigroup length} on $G$ a function $L:G\to\R_+=[0,\infty[$ such that
\begin{itemize}
\item (Local boundedness) $L$ is bounded on compact subsets of $G$.
\item (Subadditivity) $L(xy)\le L(x)+L(y)$ for all $x,y$.
\end{itemize}
We call it a \textit{length} if moreover it satisfies
\begin{itemize}
\item (Symmetricalness) $L(x)=L(x^{-1})$ for all $x\in G$.
\end{itemize}

We do not require $L(1)=1$. Note also that local boundedness weakens the more usual assumption of continuity, but also include important examples like the word length with respect to a compact generating subset. See Section \ref{discussion} for further discussion. Besides, a length is called proper if $L^{-1}([0,n])$ has compact closure for all $n<\infty$.

\begin{defn}
A locally compact group $G$ has \textit{Property PL} (respectively \textit{strong Property PL}) if every length (resp. semigroup length) on $G$ is either bounded or proper.
\end{defn}

We say that an action of a locally compact group $G$ on a metric space is {\it locally bounded} if $Kx$ is bounded for every compact subset $K$ of $G$ and $x\in X$. This relaxes the assumption of being continuous. The action is {\it bounded} if the orbits are bounded. If $G$ is locally compact, the action is called {\it metrically proper} if for every bounded subset $B$ of $X$, the set $\{g\in G|B\cap gB\neq\emptyset\}$ has compact closure.

\begin{prop}
Let $G$ be a locally compact group. Equivalences:
\begin{itemize}
\item[(i)] $G$ has Property PL;
\item[(ii)] Any action of $G$ on a metric space, by isometries, is either bounded or metrically proper;
\item[(ii')] Any action of $G$ on a metric space, by uniformly Lipschitz transformations, is either bounded or metrically proper;
\item[(iii)] Any action of $G$ on a Banach space, by affine isometries, is either bounded or metrically proper.
\end{itemize}
\label{prop:eqPL}
\end{prop}

When $G$ is compactly generated, Property PL can also be characterized in terms of its Cayley graphs.

\begin{prop}
Let $G$ be a locally compact group. If $G$ has strong Property PL (resp. Property PL), then for any subset $S$ (resp. symmetric subset)  generating $G$ as a semigroup, either $S$ is bounded or we have $G=S^n$ for some $n$. If moreover $G$ is compactly generated, then the converse also holds.\label{cayleyPL}
\end{prop}

I do not know if the converse holds for general locally compact $\sigma$-compact groups. Also, I do not know any example of a locally compact group with Property PL but without the strong Property PL.

If a locally compact group is not $\sigma$-compact, then it has no proper length and therefore both Property PL and strong Property PL mean that every length is bounded. Such groups are called {\it strongly bounded} (or are said to satisfy the {\it Bergman Property}); discrete examples are the full permutation group of any infinite set, as observed by Bergman \cite{Bergman} (see also \cite{Cor}). However the study of Property PL is mainly interesting for $\sigma$-compact groups, as it is then easy to get a proper length (it is more involved to obtain a continuous proper length; this is done in \cite{St}, based on the Birkhoff-Kakutani metrization Theorem).

The main result of the paper is.

\begin{thm}
Let $\K$ be a local field (that is, a non-discrete locally compact field) and $G$ a simple linear algebraic group over $\K$. Then $G_\K$ satisfies strong Property PL.\label{main}
\end{thm}

This result was obtained by Y. Shalom \cite{Shalom} in the case of continuous Hilbert lengths, i.e. lengths $L$ of the form $L(g)=\|gv-v\|$ for some continuous affine isometric action of $G$ on a Hilbert space, with an action of a group of $\K$-rank one. Some specific actions on $L^p$-spaces were also considered in \cite{CTV}.

My original motivation was to extend Shalom's result to actions on $L^p$-spaces, but actually the result turned out to be much more general. However, even for isometric actions on general Banach spaces, we have to prove the result not only in $\K$-rank one, but also in higher rank, in which case the reduction to $\textnormal{SL}_2$ requires some careful arguments.

The first step is the case of $\textnormal{SL}_2(\K)$; it is elementary but it seems that it has not been observed so far (even for $\K=\R$).

Then with some further work, and making use of the Cartan decomposition, we get the general case. In the case of rank one, this second step is straightforward; this was enough in the case of Hilbert lengths considered in \cite{Shalom} in view of Kazhdan's Property T for simple groups of rank $\ge 2$ (which states that every Hilbert length is bounded), but not in general as there always exist unbounded lengths.

\begin{rem}
It is necessary to consider lengths bounded on compact subsets. Indeed, write $\R$ as the union of a properly increasing sequence of subfields $K_n$. (For instance, let $I$ be a transcendence basis of $\R$ over $\Q$, write $I$ as the union of a properly increasing sequence of subsets $I_n$, and define $K_n$ as the set of reals algebraic over $\Q(I_n)$.) If $G=G(\R)$ is a connected semisimple group, then $\ell(g)=\min\{n|g\in G(K_n)\}$ is an unbounded symmetric (and ultrametric) non-locally bounded length on $G$. However $\ell$ is not bounded on compact subsets and $\{\ell\le n\}$ is dense provided $G$ is defined over $K_n$, and this holds for $n$ large enough.

Also, if $G=G(\C)$ is complex and non-compact, if $\alpha$ is the automorphism of $G$ induced by some non-continuous field automorphism of $\C$, and if $\ell$ is the word length with respect to some compact generating set, then $\ell\circ\alpha$ is another example of a non-locally bounded length neither bounded nor proper.
\end{rem}

Finally, it is convenient to have a result for general semisimple groups.

\begin{prop}
Let $\K$ be a local field and $G$ a semisimple linear algebraic group over $\K$. Let $L$ be a semigroup length on $G(\K)$. Then $L$ is proper if (and only if) the restriction of $L$ to every non-compact $\K$-simple factor $G_i(\K)$ is unbounded.\label{prop:semisimple} 
\end{prop}

This proposition relies on Theorem \ref{main}, from which we get that $L$ is proper on each factor $G_i(\K)$, and an easy induction based on the following lemma, of independent interest.

\begin{lem}
Let $H\times A$ be a locally compact group. Suppose that $H=G(\K)$ for some $\K$-simple linear algebraic group over $\K$. Let $L$ be a semigroup length on $G$, and suppose that $L$ is proper on $H$ and $A$. Then $L$ is proper.\label{lem:product}
\end{lem}

Here are some more examples of PL-groups, beyond semisimple groups.

\begin{prop}
Let $K$ be a compact group, with a given continuous orthogonal representation on $\R^n$ for $n\ge 2$, so that the action on the 1-sphere is transitive (e.g. $K=\textnormal{SO}(n)$ or $K=\textnormal{SU}(m)$ with $2m=n\ge 4$). Then the semidirect product $G=\R^n\rtimes K$ has strong Property PL.\label{isomRn}
\end{prop}
\begin{prop}
Let $\K$ be a non-Archimedean local field with local ring $\mathbf{A}$. Then the group $\K\rtimes\mathbf{A}^*$ has strong Property PL.\label{qpzp}
\end{prop}

Note that the locally compact group $\K\rtimes\mathbf{A}^*$ is not compactly generated.

\section{Discussion on lengths}\label{discussion}

We observe here that our results actually hold for more general functions than lengths. Namely, call a {\it weak length} a function $G\to\R_+$ which is locally bounded and satisfies

\begin{center}\noindent \textit{(Control Axiom)} There exists a non-decreasing function $\phi:\R_+\to\R_+$ such that for all $x,y$, we have $L(xy)\le \phi(\max(L(x),L(y))$ for all $x,y$.
\end{center}

Note that every semigroup length satisfies the control axiom with $\phi(t)=2t$.
Besides, if $L,L'$ are two weak lengths on $G$, say that $L$ is coarsely bounded by $L'$ and write $L\preceq L'$ if $L\le u\circ L'$ for some proper function $u:\R_+\to\R_+$, and that $L$ and $L'$ are coarsely equivalent, denoted $L\simeq L'$, if $L\preceq L'\preceq L$. 
Here is a series of remarks concerning various definitions of lengths.

\begin{enumerate}\item\label{i} (Vanishing at 1) Let $L$ be a length (resp. weak length). Set $L'(x)=L(x)$ if $x\neq 1$ and $L'(1)=0$. We thus get another length (resp. weak length), and obviously $L$ and $L'$ are coarsely equivalent.

\item (Continuity) A construction due to Kakutani allows to replace any length by a coarsely equivalent length which is moreover continuous (see \cite[Theorem 7.2]{Hj}).

\item For every weak length $L$, there exists a semigroup length $L_1$ coarsely equivalent to $L$. The argument is as follows: we can suppose that $\phi$, the controlling function involved in the definition of weak length, is a bijection from $\R_+$ to $[\phi(0),+\infty[$ satisfying $\phi(t)\ge t+1$ for all $t$; then it makes sense to define inductively $\alpha(t)=t/\phi(0)$ for all $t\in [0,\phi(0)]$ and $\alpha(t)=\alpha(\phi^{-1}(t))+1$ for $t\ge\phi(0)$. By construction, $\alpha(\phi(t))\le \alpha(t)+1$ for all $t$. Thus $L'_1(x)=\alpha(L(x))$ satisfies the quasi-ultrametric axiom $L'_1(xy)\le\max(L'_1(x),L'_1(y))+1$, and consequently $L_1=1+L'_1$ is a semigroup length. Moreover, $\alpha$ increases to infinity, so that $L_1$ is coarsely equivalent to $L$. Note that if $L$ is symmetric, then so is $L_1$.

\item If $L$ is a length, then $L'(x)=L(x)+L(x^{-1})$ is a symmetric length. We have ($L$ bounded) $\Leftrightarrow$ ($L'$ bounded) and ($L$ proper) $\Rightarrow$ ($L'$ proper), but $L'$ can be proper although $L$ is not. In particular, they are not necessarily coarsely equivalent; when it is the case, $L$ is called {\it coarsely symmetric}. For instance, the semigroup word length in $\Z$ with respect to the generating subset $\{n\ge -1\}$ is not coarsely symmetric.
\end{enumerate}

It is well-known that a locally compact group is $\sigma$-compact (i.e. a countable union of compact subsets) if and only if it possesses a proper length. Trivially, this is a sufficient condition. Let us recall why it is necessary: let $(K_n)$ be a sequence of compact subsets covering $G$; we can suppose that $K_1$ has non-empty interior. Define by induction $M_1=K_1$ and $M_n$ as the set of products of at most $2$ elements in $M_{n-1}\cup K_n$. Then $L(g)=\inf\{n|g\in M_n\}$ satisfies the quasi-ultrametric axiom $L(xy)\le\max(L(x),L(y))+1$ and is symmetric and proper.

\section{Elementary results on lengths}

\begin{lem}
Let $G$ be a locally compact group and $K$ a compact normal subgroup. Then $G$ has Property PL if and only if $G/K$ has Property PL.\label{compact_kernel}
\end{lem}
\begin{proof}
The forward implication is trivial. Conversely if $G/K$ has Property PL and $L$ is a length on $G$, then $L'(g)=\sup_{k\in K}L(gk)$ is a length as well, so is either bounded or proper, and $L\le L'\le L+\sup_KL$, so $L$ is also either bounded or proper.
\end{proof}

\begin{lem}
Suppose that $G$ has three closed subsets $K,K',D$ with $K,K'$ compact, and $G=KDK'$. Then a length on $G$ is bounded (resp. proper) if and only its restriction to $D$ is so.\label{cocompact_subset}
\end{lem}
\begin{proof}
Suppose that a length $L$ on $G$ is proper on $D$. Let $(g_n)$ in $G$ be bounded for $L$. Write $g_n=k_nd_n\ell_n$ with $(k_n,d_n,\ell_n)\in K\times D\times K'$. Then $L(d_n)$ is bounded. As $L$ is proper on $D$ and bounded on $K$ and $K'$, it follows that $(d_n)=(k_n^{-1}g_n\ell_n^{-1})$ is bounded; therefore $(g_n)$ is bounded as well. So $L$ is proper on all of $G$. The case of boundedness is even easier.
\end{proof}

As a consequence we get

\begin{lem}
Let $G$ be a locally compact group and $H$ a cocompact subgroup. If $H$ has (strong) Property PL, then $G$ also has (strong) Property PL.\qed\label{cocompact}
\end{lem}

The converse is not true, even when $H$ is normal in $G$, in view of Proposition \ref{isomRn}.

\begin{proof}[Proof of Propositions \ref{isomRn} and \ref{qpzp}]
Let $L$ be a semigroup length on $G$. If $L$ is not proper, then there exists an unbounded sequence $(a_i)$ in $\R^n$ with $L(a_i)\le M$ for some $M<+\infty$ independent on $i$. Using transitivity of $K$, if $M'=M+2\sup_KL$, then for every $i$, the length $L$ is bounded by $M'$ on the sphere $S(a_i)$ of radius $a_i$ centered at $0$. As every element of the ball $D(a_i)$ of radius $a_i$ centered at $0$ is the sum of two elements of $S(a_i)$, it follows that $L$ is bounded by $2M'$ on $B(a_i)$. As $a_i\to\infty$, $L$ is bounded on $\R^n$, and hence $L$ is bounded on all of $G$.

Proposition \ref{qpzp} is proved in an analogous way, using the trivial fact that in any non-Archimedean local field $\K$, for any $m\ge n$, any element of valuation $m$ is sum of two elements of valuation $n$.
\end{proof}

\begin{proof}[Proof of Proposition \ref{cayleyPL}]
Define $L(g)$ as the least $n$ such that $g\in S^n$ and observe that $L$ is bounded on compact subsets because by the Baire category  theorem, $S^k$ has non-empty interior for some $k$. If $S$ is symmetric, then $L$ is a length. So the assumption implies that either $L$ is proper (and hence $S$ is bounded) or $L$ is bounded (and hence $S^n=G$ for some $n$).

Conversely, suppose that $G$ is compactly generated and the condition holds. Let $L$ be a non-proper semigroup length (resp. length) on $G$. Set $S_n=L^{-1}([0,n])$, which is symmetric if $L$ is a length. By non-properness, there exists $n_0$ such that $S_{n_0}$ is unbounded. As $G$ is compactly generated, $S_n$ generates $G$ for some $n\ge n_0$. Then, by assumption, every element of $G$ is product of a bounded number of elements from $S=S_n$. By subadditivity, this implies that $L$ is bounded on $G$.
\end{proof}

\begin{proof}[Proof of Proposition \ref{prop:eqPL}]
\noindent(ii')$\Rightarrow$(ii) is trivial.
\noindent(i)$\Rightarrow$(ii') Let $G$ act on the non-empty metric space by $C$-Lipschitz maps, and define $L(g)=d(x_0,gx_0)$ for some $x_0$ in $X$. Then $L$ satisfies the inequality $L(gh)\le L(g)+CL(h)$ for all $g$, $h$. By the remarks at the beginning of this Section \ref{discussion}, $L$ is a weak length, so is coarsely equivalent to a length. So $L$ is either proper or bounded.
\noindent(ii)$\Rightarrow$(i) This follows from the fact that any length vanishing at 1, is of the form $d(x_0,gx_0)$ for some isometric action of $G$ on a metric space, and \ref{i}.~in Section \ref{discussion}.

Of course (ii) implies (iii). The converse follows from the construction in \cite[Section~5]{NP}: every metric space $X$ embeds isometrically into an affine Banach space $B(X)$, equivariantly, i.e. so that any isometric group action on $X$ extends uniquely to an action by affine isometries on $B(X)$.
\end{proof}

\section{Lengths on semisimple groups}

Let us now proceed to the proof of Theorem \ref{main}.

\subsection{Lengths on the affine group}

Let $\K$ be a local field, and $D$ a cocompact subgroup of $\K^*$.

\begin{prop}
Let $L$ be a symmetric length on $\K\rtimes D$. If $L$ is non-proper on $D$, then $L$ is bounded on $\K$.\label{affine}
\end{prop}
\begin{proof}
Fix $W$ a compact neighborhood of $1$, so that $L$ is bounded by a constant $M$ on $W$. Suppose that the length $L$ is not proper on $D$: there exists an unbounded sequence $(a_n)$ in $D$ such that $L(a_n)$ is bounded by a constant $M'$. Let $u$ be any element of the subgroup $\K$. Replacing some of the $a_n$ by $a_n^{-1}$ if necessary, we can suppose that $a_nua_n^{-1}\to 1$ (we use that $L$ is symmetric). Then for $n$ large enough, $w_n=a_nua_n^{-1}\in W$, on which $L$ is bounded by $M$. Writing $u=a_n^{-1}w_na_n$, we obtain that $L(u)\le M+2M'$.
\end{proof}

\begin{rem}\label{nsymaff}
Proposition \ref{affine} is false for semigroup lengths. Indeed, the subset
$$\{(x,\lambda)\in\K\rtimes D:|x|\le 1,0<|\lambda|\le M\}$$
generates $\K\rtimes D$ provided $M$ is large enough; the corresponding semigroup word length is obviously non-proper, and is easily checked to be unbounded on $\K$.
\end{rem}

\begin{rem}
Proposition \ref{affine} still holds if the normal subgroup $\K$ is replaced by a finite-dimensional $\K$-vector space, $D$ acting by scalar multiplication.
\end{rem}

\subsection{Case of $\textnormal{SL}_2$}

Denote $G=\textnormal{SL}_2(\K)$, and $D$, $U$, and 
$K$ the set of diagonal, unipotent, and orthogonal matrices in $G$. Let $L$ be any semigroup length on $G$.

We have a Cartan decomposition $G=KDK$, which implies by Lemma \ref{cocompact_subset} that boundedness and properness of the length $L$ on $G$ can be checked on $D$.

The matrix $
M=\left(\begin{matrix}
0 & -1   \\
1  & 0   
\end{matrix}\right)$
conjugates any matrix in $D$ to its inverse. It follows that $L(g)+L(g^{-1})$ is equivalent to $L$ on $D$, and hence on all of $G$. In other words, we can suppose that the length $L$ is symmetric.

So if $L$ is non-proper, then $L$ is bounded on $U$ by Proposition \ref{affine}. Similarly, $L$ is bounded on $U^t$, the lower unipotent subgroup of $G$ (this also follows from the fact that $U^t$ is conjugate to $U$ by $M$). As every element of $G$ is product of four elements in $U\cup U^t$, we conclude that $L$ is bounded on $G$.

\subsection{Reduction to $G$ simply connected}\label{redgsc}

Let $H\to G$ be the (algebraic) universal covering of $G$. Then the map $H_\K\to G_\K$ has finite kernel and cocompact image. Therefore, by Lemmas \ref{cocompact} and \ref{compact_kernel} strong Property PL for $G_\K$ follows from strong Property PL for $H_\K$.

So we can assume $G$ algebraically simply connected, and it will be convenient and harmless to identify $G$ with $G_\K$. Let $d\ge 1$ be the $\K$-rank of the simply connected $\K$-simple group $G$ and $D$ be a maximal split torus in $G$. The Cartan decomposition tells us that there exists a compact subgroup $K$ of $G$ such that $G=KDK$ (in the case of Lie groups, see \cite[Chap.~IX~1.]{Hel}; in the non-Archimedean case, see Bruhat-Tits \cite[Section 4.4]{BrT}).

So the proof consists in proving that if a semigroup length $L$ on $G$ is not proper on $D$, then it is bounded.

\subsection{Rank one}

This case is not necessary for the general case but we wish to point out that then the conclusion is straightforward.
Indeed, if $G$ is such a group, then its subgroup $D$ is contained in a subgroup isomorphic to $\text{SL}_2(\K)$ or $\textnormal{PSL}_2(\K)$ and therefore every length on $G$ is either proper or bounded on $D$.

\subsection{General case}

Remains the case of higher rank groups.
Let $W$ be the relative Weyl group of $G$ with respect to $D$, that is normalizer of $D$ (modulo its centralizer). Let $D^\vee\simeq\Z^d$ be the group of multiplicative characters of $D$, that is $\K$-defined homomorphisms from $D\simeq {\K^*}^d$ to the multiplicative group $\K^*$.

Then by \cite[Corollary~5.11]{BoT}, the relative root system is irreducible, so that by \cite[Chap.~V.3, Proposition 5(v)]{Bk}, the action of $W$ on $D^\vee\otimes_\Z\R$ is irreducible.

If $u$ is a function $D\to\R_+$, we say that a sequence $(a_n)$ in $D$ is $u$-bounded if $(u(a_n))$ is bounded.

Let $\Gamma\subset D^\vee$ be the set of $\alpha\in D^\vee$ such that every $L$-bounded sequence $(a_n)$ in $D$ is also $v\circ\alpha$-bounded, where $v(\lambda)=\log|\lambda|$ by definition. Then $\Gamma$ is a subgroup of $D^\vee$. It is easy to check that $D^\vee/\Gamma$ is torsion-free and that $\Gamma$ is $W$-invariant. On the other hand, by irreducibility, either $\Gamma=\{0\}$ or $\Gamma$ has finite index in $D^\vee$. As $D^\vee/\Gamma$ is torsion-free, this means that either $\Gamma=\{0\}$ or $\Gamma=D^\vee$.

Suppose that $L$ is not proper. Then there exists an sequence $(a_n)$ in $D$ which is $L$-bounded but not bounded. So there exists $\alpha\in D^\vee$ such that $v\circ\alpha(a_n)$ is unbounded. It follows that $\Gamma\neq D^\vee$. So $\Gamma=\{0\}$. In particular, for every relative root $\alpha$, there exists a sequence $(a_n)$ which is $L$-bounded but not $\alpha$-bounded. The argument of $\textnormal{SL}_2$ implies that $L$ is bounded on $U_\alpha$, and therefore for any root $\alpha$, $L$ is bounded on $D_{\alpha}=[U_\alpha,U_{-\alpha}]$. As any element of $D$ is a product of $d$ elements in $\bigcup D_\alpha$, we obtain that $L$ is bounded on $D$.

\section{Auxiliary results}

\begin{proof}[Proof of Lemma \ref{lem:product}]
By the argument of Lemma \ref{redgsc}, we can suppose that $G$ is simply connected.

Let $L$ be a length on $H\times A$, and suppose that $L$ is proper on both $H$ and $A$. Suppose that $L$ is not proper. Then there exists a sequence $(h_n,a_n)$ tending to infinity in $H\times A$ so that $L(h_n,a_n)$ is bounded. As $L$ is bounded on compact subsets and is proper in restriction to the factor $A$, the sequence $(h_n)$ tends to infinity. By Lemma \ref{lem:commutator} below, there exist bounded sequences $(k_n)$, $(k'_n)$ and $u$ in $G(\K)$ such that, writing $d_n=k_nh_nk'_n$, the sequence of commutators $([d_n,u])$ is unbounded.
Note that $L(d_n,a_n)$ is bounded as well.

Suppose that $L(d_n^{-1},a_n^{-1})$ is bounded (this holds if $L$ is assumed coarsely symmetric). Now $L([(d_n,a_n),(u,1)])=L([d_n,u],1)$ is bounded. But this contradicts properness of the restriction of $L$ to $H$.

If $L(d_n^{-1},a_n^{-1})$ is not assumed bounded, we can go on as follows. First note that the proof of Lemma \ref{lem:commutator} provides $(d_n)$ as a sequence in the maximal split torus $D$, and we assume this. 
If $W$ denotes the Weyl group of $D$ in $H$, then for every $d\in D$ the element $\prod_{w\in W}wdw^{-1}$ of $D$ is fixed by $W$, so is trivial.

Now the sequence $$L\left(\prod_{w\in W}(w,1)(d_n,a_n)(w^{-1},1)\right)=L(1,a_n^{|W|})$$
is bounded. Therefore, by properness on $\{1\}\times A$, the sequence $\kappa_n=a_n^{|W|}$ is bounded. Thus, the sequence $L(1,\kappa_n^{-1})$ is bounded.
Now the sequence $$L\left(\prod_{w\in W-\{1\}}(w,1)(d_n,a_n)(w^{-1},1)\right)=L((d_n^{-1},a_n^{-1})(1,\kappa_n))$$
is bounded in turn, so $L(d_n^{-1},a_n^{-1})$ is bounded, and this case is settled.
\end{proof}

\begin{lem}
Let $\K$ be a local field and $G$ a simple simply connected linear algebraic group over $\K$. Let $(g_n)$ be an unbounded sequence in $G(\K)$. Then there exist bounded sequences $(k_n)$, $(k'_n)$ and $u$ in $G(\K)$ such that the sequence of commutators $([k_ng_nk'_n,u])$ is unbounded. 
\label{lem:commutator}\end{lem}
\begin{proof}
By the Cartan decomposition (see Paragraph \ref{redgsc}), we first pick $(k_n)$ and $(k'_n)$ such that $a_n=k_ng_nk'_n$ belongs to $D$, the maximal split torus. There exists one weight $\alpha$ such that $\alpha(a_n)$ is unbounded. Fix $u$ in the unipotent subgroup $G_\alpha$. Then $a_nua_n^{-1}$ is unbounded, so $[a_n,u]$ is unbounded as well. 
\end{proof}

Proposition \ref{prop:semisimple} follows from Lemma \ref{lem:product} when $G$ is a direct product of simple groups. The general case follows by passing to the (algebraic) universal covering $\tilde{G}$ of $G$, as $\tilde{G}(\K)$ maps to $G(\K)$ with finite kernel and cocompact image.

\baselineskip=16pt


\begin{thebibliography}{KM98b}

\bibitem[Be]{Bergman} G. Bergman. Generating infinite symmetric groups. Bull. London Math. Soc. 38 429-440, 2006.

\bibitem[Bk]{Bk} N. Bourbaki. \'El\'ements de Math\'ematique. Groupes et alg\`ebres de Lie. Chap 4-6. Hermann, 1968.

\bibitem[BoT]{BoT} A. Borel, J. Tits. \newblock  
Groupes r\'eductifs.
\newblock Publ. Math. Inst. Hautes \'Etudes Sci. 27, 55-151, 1965.


\bibitem[BrT]{BrT} F. Bruhat, J. Tits. \newblock 
Groupes r\'eductifs sur un corps local. I, Donn\'ees radicielles valu\'ees.
\newblock Publ. Math. Inst. Hautes \'Etudes Sci. 41, 5-251, 1972.

\bibitem[C]{Cor} Y. de Cornulier, Strongly bounded groups and infinite powers of finite groups. Comm. Algebra 34, 2337-2345, 2006.

\bibitem[CTV]{CTV} Y. de Cornulier, R. Tessera, A. Valette.
Isometric group actions on Banach spaces and representations vanishing at infinity.
Transform. Groups 13, no. 1, 125-147, 2008.

\bibitem[He]{Hel} S. Helgason. Differential Geometry, Lie Groups, and Symmetric Spaces. Academic Press, New York 1978. 

\bibitem[Hj]{Hj} G. Hjorth. Classification and orbit equivalence relations. Mathematical Surveys and 
Monographs, 75. American Mathematical Society, Providence, RI, 2000. 

\bibitem[NP]{NP} L. Nguyen Van Th\'e, V. Pestov. Fixed point-free isometric actions of topological groups on Banach spaces. Bull. Belg. Math. Soc. Simon Stevin 17 (2010), no. 1, 29--51.

\bibitem[Sh]{Shalom} Y. Shalom. Rigidity, unitary representations of 
semisimple groups, and fundamental groups of manifolds with rank one 
transformation group. Annals of Math. 152 113-182, 2000.

\bibitem[St]{St} R. Struble. Metrics in locally compact groups. Compos. Math. 28(3), 217-222, 1974.

\end{thebibliography}
\end{document}